\documentclass{article}
\usepackage{graphicx} 

\usepackage[a4paper, top=3cm,bottom=3cm,left=2.5cm,right=3.5cm,marginparwidth=1.75cm]{geometry}
\usepackage{amsthm,amsmath,amssymb}
\usepackage{mathrsfs}
\usepackage{indentfirst}
\usepackage{amsfonts}
\usepackage{authblk}

\usepackage{amsmath}
\usepackage{cleveref}
\usepackage{enumerate}
\usepackage{enumitem}
\usepackage{indentfirst}

\usepackage{ulem}

\usepackage{xcolor}

\newtheorem{theorem}{Theorem}
\newtheorem{lemma}{Lemma}
\newtheorem{claim}{Claim}

\newtheorem*{Proof}{Proof}
\newtheorem{definition}{Definition}

\newtheorem{observation}{Observation}

\newenvironment{Proof*}
  {\begin{Proof}}
  {\end{Proof}}
\title{Planar graphs without cycles of length 4 or 5 are $(7m:2m)$-DP-colorable}
\author{Xiaoyan Xu \qquad  Xuding Zhu\thanks{Grant number: NSFC 12371359} \\
School of Mathematical Sciences, Zhejiang Normal University}
\date{\today}

\begin{document}

\maketitle
\begin{abstract}
  It was conjectured by Steinberg in 1976  that   planar graphs without cycles of length 4 or 5 are 3-colorable. This conjecture  attracted a substantial amount of attention and was finally refuted by  Cohen-Addad,  Hebdige, Kr\'{a}l',   Li and Salgado in 2017. Although Steinberg's conjecture is settled, coloring of this family of graphs, as well as some other families of planar graphs forbidding certain cycle lengths have been attracting a lot of recent attention and many challenging problems remain open. One problem of interest is multiple coloring and multiple list coloring of this family of graphs.  It was proved by Dv\v{o}r\'{a}k and Hu that 
  planar graphs without cycles of length 4 or 5  are $(11,3)$-colorable, and this result was improved by Wang, who proved  that graphs in this family are $(7:2)$-colorable. On the other hand, it was proved by Xu and Zhu that for every positive integer $m$, there is a graph in this family which is not $(3m + \lfloor \frac{m-1}{12} \rfloor, m)$-choosable. 
  In this paper, we prove that for any positive integer $m$, graphs in this family are $(7m:2m)$-DP-colorable, and hence $(7m,2m)$-choosable.  
\end{abstract}

\section{Introduction}

One problem of interest in planar graph coloring is which planar graphs are 3-colorable, and which planar graphs are 3-choosable.  For a sequence $k_1,k_2, \ldots, k_q$ of positive integers, we denote by $\mathcal{P}_{k_1,k_2, \ldots, k_q}$ the family of connected planar graphs without cycles of length $k_i$ for $i=1,2,\ldots, q$.
The classical Gr\"{o}tzsch Theorem \cite{MR116319} asserts graphs in $\mathcal{P}_{3}$  are 3-colorable.  In 1976, Steinberg conjectured that graphs in $\mathcal{P}_{4,5}$ are 3-colorable   (see \cite{MR3004485}). This conjecture attracted considerable attention. As an approach to this conjecture,  Erd\H{o}s proposed the problem in 1991 to determine the smallest integer  $k$ such that every  graph in $\mathcal{P}_{4,5, \ldots, k}$ is  3-colorable. It was proved by Abbott and Zhou  that $k \le 11$  \cite{MR1148923}, and this bound was improved by Sanders and Zhao  to $k \le 9$ \cite{MR1326930}, and further improved by Borodin, Glebov, Raspaud and Salavatipour  to $k \le 7$ \cite{MR2117940}. Many other relaxations or variations of Steinberg's conjecture have been proved in the literature. For example, it is known   that  for $4 < i < j < 10$  and $(i,j) \notin \{(5,6), (7,8), (5,9),(8,9)\}$,   graphs in $\mathcal{P}_{4,i,j}$  are 3-colorable (see \cite{KangJinZhu} for detailed references). 
In 2017,  Cohen-Addad,  Hebdige, Kr\'{a}l',   Li and Salgado \cite{MR3575214} proved that 
Steinberg's conjecture is false. Now it remains a challenging open question whether graphs in $\mathcal{P}_{4,5,6}$  are 3-colorable.  

For list coloring, Thomassen's classical result asserts that planar graphs of girth at least 5 are 3-choosable.  Erd\H{o}s' problem is also naturally extended to the list version: Determine the smallest integer $\ell$ such that graphs in $\mathcal{P}_{4,5,\ldots, \ell}$  are 3-choosable. It was proved by Voigt \cite{Voigt2007}  that $\ell \ge 6$, and proved by Dv\v{o}r\'{a}k and Postle \cite{dvovrak2018correspondence} that $\ell \le 8$. Moreover, it is known   that  graphs in $\mathcal{P}_{4,7,9},\mathcal{P}_{3,6,7,8},\mathcal{P}_{4,5,6,9},\mathcal{P}_{4,6,8,9}$  are 3-DP-colorable, and hence $3$-choosable  (see \cite{KangJinZhu} for detailed references).  

Multiple coloring and multiple list coloring of graphs allow us to explore the colorability of graphs in a finer scale.  As Steinberg's conjecture is disproved, a natural problem is explore multiple coloring of graphs in $\mathcal{P}_{4,5}$.   

    \begin{definition}
      Given a graph $G$ and a function $g \in \mathbb{N}^G$, a \textit{$g$-fold coloring} of $G$ is a mapping $\phi$ which assigns to each vertex $v$ of $G$ a set $\phi(v)$ of $g(v)$ colors, so that adjacent vertices receive disjoint color sets. For a mapping $f: V(G) \to \mathbb{N}$, an $f$-list assignment of $G$ is a mapping $L$ that assigns to each vertex $v$ a set $L(v)$ of $f(v)$ colors. An $(L,g)$-coloring of $G$ is a $g$-fold coloring $\phi$ of $G$ such that $\phi(v) \subseteq L(v)$ for each vertex $v$. We say $G$ is $(f, g)$-choosable if $G$ has an $(L,g)$-coloring for each $f$-list assignment $L$. For positive integers $a,b$, we say $G$ is $(a,b)$-choosable if $G$ is $(f,g)$-choosable for the constant functions $f(v)=a$ and $g(v)=b$ for all $v$. 
    \end{definition}

  If $L(v)=[a]=\{1,2,\ldots, a\}$ for each vertex $v$, then an $(L,b)$-coloring of a graph $G$ is called an $(a,b)$-coloring  of $G$. The \textit{  fractional chromatic number} of $G$ is defined as $$\chi_f(G) = \inf \{ \frac ab: G \text{ is $(a,b)$-colorable}\}.$$ 
 The \textit{fractional choice number} of $G$ is defined as 
 $$ch_f(G) = \inf\{ \frac ab: G \text{ is $(a,b)$-choosable}\}.$$ 
 It was proved by Alon,Tuza and Voigt \cite{AlonTuzaVoigt} that if $G$ is $(a,b)$-colorable, then for some positive integer $m$, $G$ is $(am,bm)$-choosable. Hence $ch_f(G)=\chi_f(G)$ for all $G$. So $ch_f(G)$ is not a new parameter. The concept of \textit{strong fractional choice number} $ch_f^*(G)$ of a graph $G$, introduced by Zhu in 2017 \cite{ZhuStrongfractionalchoicenumber}, is defined as follows:
 $$ch_f^*(G) = \inf\{r: G \text{ is $(a,b)$-choosable for all integers $a, b$ with $a/b \ge r$}\}.  $$

For a family $\mathcal{G}$ of graphs, let 
$$\chi_f( \mathcal{G}) = \sup \{ \chi_f(G): G \in \mathcal{G}\}  \text{ and } ch_f^*(\mathcal{G}) = \sup \{ ch_f^*(G): G \in \mathcal{G}\}.$$

 Multiple coloring and multiple list coloring of graphs in $\mathcal{P}_{4,5}$ are studied in a few papers. It was proved by Dv\v{o}r\'{a}k and Hu \cite{DvorakHu11/3} that   every graph $G \in \mathcal{P}_{4,5}$ is $(11,3)$-colorable. This result was improved by Wang \cite{WangYingQian} who proved that every graph $G \in \mathcal{P}_{4,5}$ is $(7,2)$-colorable. The current known upper bound for the  strong fractional choice number is $4$, and for the lower bound, it was proved by Xu and Zhu \cite{XuZhuStrongfractionalchoicenumber} that for any positive integer $m$, there is a  graph $G \in \mathcal{P}_{4,5}$  which is not  $(3m+ \lfloor \frac {m-1}{12} \rfloor, m)$-choosable. This implies that $ch_f^*(\mathcal{P}_{4,5}) \ge 3 + \frac 1{12}$. In this paper, we prove that for any positive integer $m$, every planar graph $G \in \mathcal{P}_{4,5}$ is $(7m,2m)$-choosable, and hence $ch_f^*(\mathcal{P}_{4,5}) \le 3 + \frac 12$.

Our result is proved in the setting of multiple DP-coloring of graphs.  
The concept of DP-coloring (also known as correspondence coloring) is a generalization of list coloring introduced by Dvo{\v{r}}{\'a}k and Postle in \cite{dvovrak2018correspondence}. 

\begin{definition}
    Let $G$ be a graph. A \textit{ cover} of $G$ is a pair $(L,M)$, where $L=\left \{ L(v): v\in V(G) \right \} $ is a family of pairwise disjoint sets, and $M=\left \{ M_e: e\in E(G) \right \} $ is a family of matchings such that for each edge $e=uv$, $M_e$ is a   matching between sets $L(u)$ and $L(v)$. For a function $f:V(G)\to \mathbb{N}$, we say $(L,M)$ is a $f$-cover of $G$ if $\left | L(v) \right | = f(v)$ for each vertex $v\in V(G)$.   
\end{definition}

\begin{definition}
    Given a cover $\mathcal{H}=(L,M)$ of a graph $G$, an $\mathcal{H}$-coloring of $G$ is a mapping $\phi: V(G) \to \cup_{v\in V(G)}L(v)$ such that for each vertex $v\in V(G), \phi(v) \in L(v)$, for each edge $e=uv\in E(G), \phi(u) \phi(v)\ne E(M_e)$. We say $G$ is \textit{$\mathcal{H}$-colorable} if it has an $\mathcal{H}$-coloring. For a function $f: V(G) \to \mathbb{N} $, we say $G$ is \textit{$f$-DP-colorable} if $G$ is  $\mathcal{H}$-colorable for every $f$-cover $\mathcal{H}$ of $G$.
    The DP-chromatic number of $G$ is defined as 
    $$\chi_{DP}(G) = \min\{k: G \text{ is $k$-DP-colorable}\}.$$
\end{definition}

The concept of $f$-DP-colorable is a generalization of $f$-choosable. Given an $f$-list assignment $F$ of $G$, let $\mathcal{H}=(L,M)$ be the $f$-cover of $G$ defined as follows: For each vertex $v$, $L(v) = \{(v,c): c \in F(v)\}$, and for each edge $e=uv$ of $G$, $M_e= \{ (u,c)(v,c): c \in F(u) \cap F(v)\}$. It is easy to see that an $\mathcal{H}$-coloring of $G$ is equivalent to an $F$-coloring of $G$. Therefore $ch(G) \le \chi_{DP}(G)$ for any graph $G$. It is known \cite{BernshteinYuKostochkaPron} that the difference $\chi_{DP}(G)-ch(G)$ can be arbitrarily large.

Multiple DP-coloring of graphs was first studied in \cite{bernshteyn2020fractional}. 

\begin{definition}
    Assume $\mathcal{H}=(L,M)$ is a cover of $G$ and $g\in \mathbb{N}^G$. An $(\mathcal{H},g)$-coloring of $G$ is a mapping $\phi$ that assigns to each vertex $v  $ a subset  $\phi(v)$ of $ L(v)$ of size $\left | \phi(v) \right |=g(v) $ and for each edge $e=uv$, for each $c \in \phi(u), c' \in \phi(v)$, $cc' \notin M_e$. We say $G$ is $(\mathcal{H},g)$-colorable if there exists an $(\mathcal{H},g)$-coloring of $G$. We say  $G$ is $(f,g)$-DP-colorable if for any $f$-cover $\mathcal{H}$ of $G$, $G$ is $(\mathcal{H},g)$-colorable. If $f,g\in \mathbb{N}^G$ are constant maps with $g(v)=b$ and $f(v)=a$ for all $v\in V(G)$, then  $(f,g)$-DP-colorable is called $(a,b)$-DP-colorable. The \textit{fractional DP-chromatic number} of $G$ is defined as
    $$\chi^*_{DP}(G) = \inf\{ \frac ab: G \text{ is $(a,b)$-DP-colorable}\}.$$
    The \textit{strong fractional DP-chromatic number} of $G$ is defined as
    $$\chi^{**}_{DP}(G) = \inf\{ r: G \text{ is $(a,b)$-DP-colorable for any $a,b$ with $\frac ab \ge r$}\}.$$ 
    For a family $\mathcal{G}$ of graphs,  $$\chi^{*}_{DP}(\mathcal{G}) = \sup\{\chi^{*}_{DP}(G):G \in \mathcal{G}\} \text{ 
 and } \chi^{**}_{DP}(\mathcal{G}) = \sup\{\chi^{**}_{DP}(G):G \in \mathcal{G}\}.   $$
\end{definition}

Similarly,   if a graph $G$ is  $(f,g)$-DP-colorable, then $G$ is $(f,g)$-choosable. Hence $ch_f^*(G) \le \chi_{DP}^{**}(G)$ for any graph $G$.  This paper proves the following theorem.

\begin{theorem}
    \label{Th:1}
    For every graph $G \in \mathcal{P}_{4,5}$, for any positive integer $m$, $G$ is $(7m,2m)$-DP-colorable. Consequently, $$\chi^{**}_{DP}(\mathcal{P}_{4,5}) \le 3 + \frac 12.$$
\end{theorem} 
 
The advantage of working in the setting of DP-coloring, instead of list coloring, is that when proving the theorem by induction, we can identify non-adjacent vertices in some cases. This is the motivation for the introduction of DP-coloring in \cite{dvovrak2018correspondence}, and play critical role in the proofs in this paper.  

The following terminology will be used in this paper. 
A plane graph $G$ is denoted as $G=(V,E,F)$, where $V,E,F$ are  the vertex set, the edge set and the face set of $G$, respectively. Denote by $d_G(v)$  the degree of vertex $v$ in $G$. A $k$-vertex ($k^{+}$-vertex, $k^{- }$-vertex, respectively) is a vertex of degree $k$ (at least $k$, at most $k$, respectively). The degree of a face is the length of the boundary walk of the face. A $k$-face, $k^{- }$-face or a $k^{+ }$-face is a face of degree $k$ (at most $k$ or at least $k$, respectively). If $v_1,v_2,\ldots,v_k$ are the vertices of a face $f$ in a cyclic order, then we write $f=[v_1v_2 \ldots v_k]$.

 A  vertex or an edge is called \textit{triangular} if it is incident with a triangle, and  a vertex $v$ is called \textit{$k$-triangular} if it is incident with exactly $k$ triangles. 
 If $uv$ is an edge and $u$ an isolated vertex in $G[N_G(v)]$, then we say $u$ is an \textit{isolated neighbor} of $v$.  Otherwise $u$ is an non-isolated neighbor of $v$. Note that $u$ is an non-isolated neighbor of $v$ if and only if $uv$ is a triangular edge.

 Let $f_0$ denote unique unbounded face of $G$. Call a vertex $v$ in $G$ \textit{external} if $v\in b(f_0)$; \textit{internal} otherwise. For a cycle $C$ of $G$, the length of $C$ is the number of edges of $C$ and is denoted by $|C|$. Call $C$ a \textit{$k$-cycle} if $|C|=k$. Denote by ${\rm int}(C)$ (resp. ${\rm ext}(C)$) the set of   vertices lie strictly inside (resp, lie strictly outside) of $C$. Let ${\rm int}[C]=G-{\rm ext}(C)$ and ${\rm ext}[C]=G-{\rm int}(C)$. Call $C$ a \textit{separating} cycle of $G$ if ${\rm int}(C)\ne \emptyset  \ne {\rm ext}(C)$.  A path with  two ends  external and other vertices internal is called a \textit{splitting} path of $G$. 

 Call an internal vertex $v$ of $G$ \textit{poor}, if $v$ is either a triangular 3-vertex and has  an internal isolated 3-neighbor, or a triangular 4-vertex and has two internal isolated 3-neighbors.

Given  a cover $\mathcal{H}=(L,M)$ of $G$, we 
  view $M$ as a graph with vertex set $\cup_{v \in V(G)}L(v)$, and with edge set 
 $\cup_{e \in E(G)}M_e$. Then for $g \in \mathbb{N}^G$, for an $(\mathcal{H}, g)$-coloring $\phi$ of $G$, $\cup_{x \in V}\phi(x)$  is an independent set   of $M$ with $|\cup_{x \in V}\phi(x) \cap L(v)| = |\phi(v)| = g(v)$ for each vertex $v$.  The set $\phi(v)$, which is set of vertices of $M$, is called the set of colors assigned to $v$.
 
 For a color $c \in L(v)$, $N_M(c)$ denotes the set of neighbors of $c$ in the graph $M$. I.e.,  
$$N_M(c) = \{c': \exists u \in N_G(v), \text{ such that }  c' \in L(u) \text{ and }  cc' \in M_{uv}\}.$$
Thus if $\phi$ is an $(\mathcal{H}, g)$-coloring   of $G$ and $c \in \phi(v)$, then no colors in $N_M(c)$ is assigned to any of the neighbors of $v$. 


\section{Extending a coloring of the boundary cycle to the whole graph}

For the purpose of using induction, instead of proving Theorem \ref{Th:1} directly, we shall prove a stronger and more technical statement.  A cycle $C$ of a plane graph $G$ is called a \textit{good cycle} if $|C| \le 7$. 

\begin{theorem}
\label{Th:2}
    If $G \in \mathcal{P}_{4,5}$,  $D=b(f_0)$ is a good cycle, then for any $7m$-cover $\mathcal{H}=(L,M)$ of $G$, any $(\mathcal{H},2m)$-coloring of $G[V(D)]$ can be extended to an $(\mathcal{H},2m)$-coloring of the whole graph $G$. 
\end{theorem}

It is easy to see that Theorem \ref{Th:1} follows from Theorem \ref{Th:2}: Assume $G\in \mathcal{P}_{4,5}$. If $G$ contains no triangle, then $G$ has girth at least $6$, and hence is 2-degenerate. Then $G$ is $(6m,2m)$-DP-colorable for any integer $m$ (and hence $(7m,2m)$-DP-colorable for any integer $m$).

Assume $G$ has a triangle $C$ and $\mathcal{H}=(L,M)$ is a $7m$-cover of $G$. Let $\phi$ be an $(\mathcal{H},2m)$-coloring of $C$. By Theorem \ref{Th:2}, $\phi$ can be extended to an $(\mathcal{H},2m)$-coloring $\phi_1$ of ${\rm int}[C]$, as well as an $(\mathcal{H},2m)$-coloring $\phi_2$ of ${\rm ext}[C]$.
The union $\phi_1 \cup \phi_2$ is an $(\mathcal{H},2m)$-coloring of $G$.

\bigskip 
 The remaining part of the paper is devoted to the proof of Theorem \ref{Th:2}.

For convenience, an extension of $\phi$ to an  $(\mathcal{H}, 2m)$-coloring of $G$ or a subgraph $G-Z$ of $G$ is also denoted by $\phi$, and denote the restriction of $\phi$ to $D$ (which is the initial coloring $\phi$ of $D$) by $\phi|_D$. Thus we call $\phi$   \textit{an extension of $\phi|_D$ to $G$ (or $G-Z$)}.  

 In the proof of Theorem \ref{Th:2}, we often need to first find an $(\mathcal{H},2m)$-coloring $\phi$ of a subgraph $G-Z$ of $G$. Then extend $\phi$ to the whole graph, by coloring vertices in $G[Z]$. 
 
 For a vertex $v \in Z$, the set of available colors is  $L^{\phi}(v) = L(v) - N_M(\phi(G-Z))$ (those colors in $L(v)$ adjacent to a color in $M$ assigned to a neighbor of $v$ cannot be used on $v$ anymore). Let $\mathcal{H}^{\phi} = (L^{\phi}, M^{\phi})$, where $M^{\phi}$ is the restriction of $M$ to $\cup_{v \in X} L^{\phi}(v)$,  we need to find an $(\mathcal{H}^{\phi}, 2m)$-coloring of $G[Z]$. For this purpose, we need to show that 
 $G[Z]$ is $(f,2m)$-DP-colorable, where $f(v) = |L^{\phi}(v)|$ for $v \in Z$. 

 In Section \ref{sec-f2m}, we shall show that for some small graphs $H$, for some $f \in \mathbb{N}^H$, $H$ is $(f,2m)$-DP-colorable.

 In Section \ref{sec-preparation}, we give a sketch of the proof of Theorem \ref{Th:2}, and prove a preliminary lemma needed in later proofs.

In Section \ref{sec-property}, by using the results in Section \ref{sec-f2m},  we shall prove that certain configuartions are reducible, and hence cannot be contained in a minimal counterexample.     

In Section \ref{sec-discharging}, by using discharging method, we show that any graph $G \in \mathcal{P}_{4,5}$ contains some reducible configuration, and hence cannot be a minimal counterexample to Theorem \ref{Th:2}. 

\section{Multiple DP-coloring of some small graphs}
\label{sec-f2m}

In this section, for a family of small trees $T$, we prove that $T$ is $(f, g)$-DP-colorable, where $f,g \in \mathbb{N}^T$. These results will be used in Section \ref{sec-property} to prove some configurations are reducible.
First we show that for trees, $(f,g)$-DP-colorable is equivalent to $(f,g)$-choosable.

\begin{lemma}
    \label{lem-choosable}
    Assume $T$ is a tree and $f,g \in \mathbb{N}^T$. Then $T$ is $(f,g)$-DP-colorable if and only if $T$ is $(f,g)$-choosable. Moreover, to prove that $T$ is $(f,g)$-choosable, we may restrict to those $f$-list assignment $L$ such that for each edge $uv$ of $T$, either $L(u) \subseteq L(v)$ or $L(v) \subseteq L(u)$.
\end{lemma}
\begin{proof}
    We have observed before that for any graph $G$, $(f,g)$-DP-colorable  implies $(f,g)$-choosable. Assume $T$ is a tree and $f,g \in \mathbb{N}^T$. Assume $T$ is $(f,g)$-choosable. We shall show that $T$ is $(f,g)$-DP-colorable.
 
    Assume $\mathcal{H}=(L,M)$ is an $f$-cover of $T$. For each edge $e=uv$, if $|L(u)| \le |L(v)|$, by adding edges to $M_e$ if needed, we may assume that the matching $M_e$ saturates $L(u)$. Consider the graph $M$ with edge set $\cup_{e \in E(T)}M_e$. Each connected component $B$ of $M$ is isomorphic to a subtree of $T$. Assume the connected components of $M$ are $B_1,B_2, \ldots, B_q$. Let $L'$ be the list assignment of $T$ defined as $L'(v) = \{i: V(B_i) \cap L(v) \ne \emptyset\}$. It is obvious that $|L'(v)|=|L(v)| = f(v)$. Since $T$ is $(f,g)$-choosable, there is a $g$-fold $L'$-coloring $\phi'$ of $T$. Let $\phi(v) = \cup \{B_i \cap L(v): i \in \phi'(v)\}$. It is easy to check that $\phi$ is an $(\mathcal{H}, g)$-coloring of $T$. 
\end{proof}

In the following, given a tree $T$, functions $f,g \in \mathbb{N}^T$, and an $f$-list assignment $L$ of $T$, we shall construct an $(L,g)$-coloring $\phi$ of $T$. The coloring $\phi$ is constructed by adding colors to $\phi(x)$ for each vertex $x$ in gradually. In the process, $\phi(x)$ always denote the current set of colors assigned to $x$. So $\phi(x)$ changes in the process.  Initially, $\phi(x)= \emptyset$ for all $x \in V(T)$.

\begin{lemma}
    \label{lem-claw}
    Assume $T$ is a claw with vertex set $\{u, v_1,v_2,v_3\}$, and edge set $\{uv_1,uv_2,uv_3\}$, and $f \in \mathbb{N}^T$ is defined as $f(v_i)=3m$ for $i=1,2,3$ and $f(u)=5m$. Then $T$ is $(f, 2m)$-choosable. 
\end{lemma}
\begin{proof}
   Let $L$ be an $f$-list assignment of $T$. We construct an
    $L$-coloring $\phi$ of $T$ by gradually adding colors to $\phi(x)$ for $x \in V(T)$. Initially, 
    $\phi(x)=\emptyset$ for all $x \in V(T)$. 
    We may assume that $ L(v_i)\subseteq  L(u) $   for $i=1,2,3$. As $|L(v_1) \cup L(v_2)| \le |L(u)| =5m$, we have $$ |L(v_1)  \cap  L(v_2)  | \ge m.$$ 
    Let $A $ be an $m$-subset $  L(v_1) \cap L(v_2)$,  add $A$ to $\phi(v_i)$ for $i=1,2$.  
     
   As  $|L(u) -   (L(v_3) \cup A)| \ge m$, there is  an $m$-subset $B$ contained in $L(u) -   (L(v_3) \cup A)$. Add $B$ to $\phi(u)$. 
    
    For $i=1,2$,  $|L(v_i) - (A \cup B)| \ge m$. Choose an $m$-subset $A'_i \subseteq L(v_i) - (A \cup B)$  and add $A'_i$ to $\phi(v_i)$. 
    
    As $|L(u) - (A \cup A'_1 \cup A'_2 \cup B)| \ge m$, we choose  a subset $B'$ of $L(u) - (A \cup A'_1 \cup A'_2 \cup B)$ and add $B'$ to $\phi(u)$.

    Let $\phi(v_3)$ be a $2m$-subset of $L(v_3)-B'$. Then $\phi$ is an $(L, 2m)$-coloring of $T$.
\end{proof}

\begin{figure}[htp]
    \centering
    \includegraphics[width=13cm]{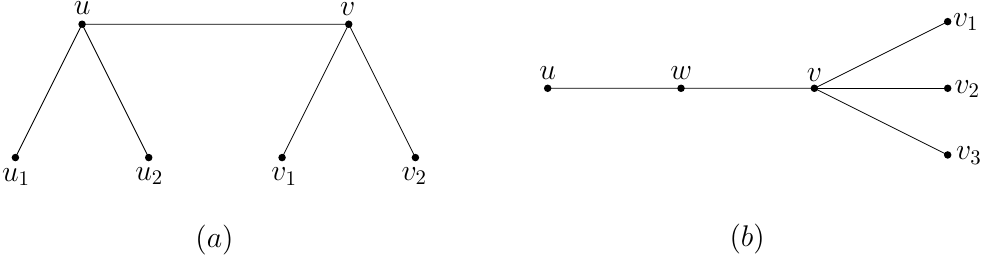}
    \caption{\label{fig:ori1} (a) The tree in Lemma \ref{lem-doubleclaw} and Lemma \ref{lem-doublestar2}; (b) The tree in Lemma \ref{lem-broom}}
\end{figure}

\begin{lemma}
\label{lem-doubleclaw}
    Let $T$ be a double claw with vertex set $\{u,v, u_1,u_2, v_1,v_2\}$ 
    and edge set $\{uv,uu_1,uu_2, vv_1, vv_2\}$. 
    Let $f(u)=f(v)=4m$ and $f(u_i)=f(v_i)=2m$ for $i=1,2$. 
    Let $g(u)=g(v)=2m$ and $g(v_i) = g(u_i) =m$ for $i=1,2$. Then $T$ is $(f,g)$-choosable. 
\end{lemma}
\begin{proof}
Let $L$ be an $f$-list assignment of $T$. Initially, we let $\phi(x)= \emptyset$ for $x \in V(T)$. 

By Lemma \ref{lem-choosable}, we may assume that $L(u)=L(v)$ and   $L(v_i), L(u_i) \subseteq L(u) $ for $i=1,2$. 

\noindent
{\bf Case 1}  $L(u_1) \cap L(u_2)| \ge m$.

Let $A$ be an $m$-subset of $L(u_1) \cap L(u_2)$, add $A$ to $\phi(u_1)$ and $\phi(u_2)$. 
 
 Let $L'(u) = L(u) - \phi(u_1)$ and  $L'(x)=L(x)$ for $x \in \{v, v_1, v_2\}$. 
 
 Let $B$ be an $m$-subset of $L'(u) \cap L(v_1)$. Add $B$ to $\phi(u)$ and $\phi(v_1)$. 
 
 Let $L''(v)=L(v)-B, L''(u)=L'(u)-B$ and $L''(v_2)=L(v_2)$. We may assume that $L''(v_2) \subseteq L''(v)$. Then $L''(u) \cap L''(v_2)$ contains an $m$-subset $C$. Add $C$ to $\phi(u)$ and $\phi(v_2)$. 
 
 Add a $2m$-subset of $L''(v)-C$ to $\phi(v)$.  Then $\phi$ is an $(L,g)$-coloring of $T$. 

\noindent
{\bf Case 2}  $L(u_1) \cap L(u_2)| < m$.

 Then $|L(u_1) \cup L(u_2)| > 3m$. Let $A$ be an $m$-subset of $(L(u_1) \cup L(u_2)) - L(v_1) $. We add $A \cap L(u_i)$ to $\phi(u_i)$. Let $B$ an $m$-subset of $L(v_2) - A$, and let $\phi(v_2) = B$. 

For $i=1,2$, let  $A_i$ be a subset of $L(u_i)-B-A$ of size $|A_i|=m-|A \cap L(u_i)|$, and add $A_i$ to $\phi(u_i)$. 

Note that $|A_1 \cup A_2| \le |A_1|+|A_2| = m$. 
Let $\phi(v_1)$ be an $m$-subset  of $L(v_2) - (A_1 \cup A_2)$. 
Now each of $u_1,u_2, v_1, v_2$ has been assigned a set of $m$ colors. Moreover, $(\phi(u_1) \cup \phi(u_2)) \cap (\phi(v_1) \cup \phi(v_2)) = \emptyset $.

Let $\phi(u)$ be a $2m$-subset of $L(u)-(\phi(u_1) \cup \phi(u_2))$ that contains $ \phi(v_2) \cup \phi(v_1)$, and let $\phi(v) =L(v) -\phi(u)$. We obtain an $(L,g)$-coloring of $T$. 
   \end{proof}

    \begin{lemma}
        \label{lem-doublestar2}
         Let $T$ be a double claw with vertex set $\{u,v, u_1,u_2, v_1,v_2\}$ and edge set $\{uv,uu_1,uu_2, vv_1, vv_2\}$. 
    Let $f(u)=f(v)=5m$ and $f(u_i)=f(v_i)=3m$ for $i=1,2$. 
      Then $T$ is $(f,2m)$-choosable, and hence $(f,2m)$-DP-olorable. 
    \end{lemma}
    \begin{proof}
  By Lemma \ref{lem-choosable}, we may assume that $L(u_i), L(v_i) \subseteq L(u)=L(v)$. 
  Since $|L(u_1) \cup L(u_2)| \le |L(u)| = 5m$, $|L(u_1) \cap L(u_2)| \ge m$. Similarly, $|L(v_1) \cap L(v_2)| \ge m$.  
  Let $A$ be an $m$-subset of $L(u_1) \cap L(u_2)$ and $B$ be an $m$-subset of $L(v_1) \cap L(v_2)$. Initially, let $\phi(x) = \emptyset$ for $x \in V(T)$. We add $A$ to $\phi(u_1)$ and $\phi(u_2)$, and add $B$ to $\phi(v_1)$ and $\phi(v_2)$. Now each of $v_1,v_2,u_1,u_2$ has been assigned a set of $m$ colors. Let 
    $L'(u_i)=L(u) - A$, $L'(v_i) = L(v)-B$ for $i=1,2$, $L'(u)=L(u)-A$ and $L'(v)=L(v)-B$. Then $L'$ is an $f'$-list assignment of $T$, where $f'(u)=f'(v)=4m$, $f'(u_i)=f'(v_i)=2m$ for $i=1,2$. Let $g'(u)=g'(v)=m$, $g'(v_i) = g'(u_i) =m$ for $i=1,2$. By Lemma \ref{lem-doubleclaw}, there is an $(L', g')$-coloring $\psi$ of $T$. Then $\phi \cup \psi$ is  an $(L, 2m)$-coloring of $T$. 
\end{proof}

\begin{lemma}
\label{lem-star5}
    Let $T$ be the tree with vertex set $\{ v,v_1,v_2,v_3,v_4,v_5\}$ and edge set $\{ vv_1,vv_2,vv_3,vv_4,vv_5\}$. Let $f\in \mathbb{N}^T$ be the function defined as $f(v)=7m$ and $f(v_i)=3m$ for $i=1,\dots,5$. Then $T$ is $(f,2m)$-choosable.
\end{lemma}
\begin{proof}
    Let $L$ be an $f$-list assignment of $T$. By Lemma \ref{lem-choosable}, we may assume that $L(v_i) \subseteq L(v)$. 
    As 
      
        $$7m \ge |\cup_{i=1}^5 L(v_i)| \ge 15m - \sum_{1\le i< j \le 5}|L(v_i) \cap L(v_j)|,$$  we conclude that there exist $i< j$ such that 
        $$|L(v_i) \cap L(v_j)| \ge \frac 45 m.$$
       
    Assume $|L(v_1)\cap L(v_2)|\ge \frac45m$. Hence   $|L(v_1)\cup L(v_2)|\le 6m-\frac45m$ and $$|L(v)\setminus (L(v_1)\cup L(v_2)|\ge \frac95m.$$ 
    
    For $i=1,2,3$, assume there are $a_i$ colors contained in exactly $i$ of the sets  $ L(v_3),  L(v_4), L(v_5) $. Then  $\sum_{i=1}^3 ia_i =9m$ and $\sum_{i=1}^3 a_i \le 7m$.  Hence $$2a_3+a_2 \ge 2m.$$ 
    
   \noindent
{\bf Case 1} $a_3 \ge \frac{4}{5}m$.

Let $A$ be a  $\frac 45 m$-subset of   $L(v_3)\cap L(v_4)\cap L(v_5)$. Let $B$ be an $m$-subset of $L(v)\setminus (L(v_1)\cup L(v_2)\cup A)$.
      We add $A$ to $\phi(v_i)$ for $i=1,2,3$, and add $B$ to $\phi(v)$. 

   For $i=3,4,5$, choose a  $\frac{6}{5}m$-subset $A_i$ of $L(v_i)-(A \cup B)$, and add $A_i$ to $\phi(v_i)$.
   
   Note that $|L(v) - (A \cup B \cup A_3 \cup A_4 \cup A_5)| \ge \frac 85 m $. Choose an $m$-subset $B'$ of $L(v) - (A \cup B \cup A_3 \cup A_4 \cup A_5)$, and add $B' $ to $\phi(v)$. For $i=1,2$, $|L(v_i) - \phi(v)| = |L(v_i)-B'| \ge 2m$. Choose a $2m$-subset of $L(v_i) - \phi(v)$ and add to $\phi(v_i)$. We obtain an $(L, 2m)$-coloring of $T$. 
   
    \noindent
{\bf Case 2}  $a_3 < \frac 45 m$.

Let $A = L(v_3) \cap L(v_4) \cap L(v_5)$.   Let $B$ be an $m$-subset of $L(v)\setminus (A\cup L(v_1) \cup L(v_2))$. We add $A$ to $\phi(v_i)$ for $i=3,4,5$, and add $B$ to $\phi(v)$.

    Since $2a_3+a_2 \ge 2m$, we have  $a_2 \ge 2(m-a_3)$.  Let $A^{\prime}$ be a $2(m-a_3)$-subset of colors that are contained in exactly two of $L(v_3),L(v_4), L(v_5)$. For $i=1,2,3$, add
    $A' \cap L(v_i)$ to $\phi(v_i)$. For $i=3,4,5$, choose an    
    arbitrary $(2m-|\phi(v_i)|)$-subset $A'_i$  colors from $L(v_i) - \phi(v_i)$, and add $A'_i$  to $\phi(v_i)$. 
    Note that $\sum_{i=3}^5 |A'_i| = 6m- 3a_3 - 4(m-a_3)=2m +a_3$. Hence $$|\phi(v_3) \cup \phi(v_4) \cup \phi(v_5)| = a_3 + 2(m-a_3)+ 2m +a_3 = 4m.$$
    
     Let $B'$ be an $m$-subset of $L(v)\setminus (\phi(v_3)\cup\phi(v_4)\cup \phi(v_5)\cup B)$. Add  $B'$ to $\phi(v)$. 

     For $i=1,2$, $|L(v_i) - \phi(v)| = |L(v_i)-B'| \ge 2m$. Choose a $2m$-subset of $L(v_i) - \phi(v)$ and add to $\phi(v_i)$. We obtain an $(L, 2m)$-coloring of $T$.
\end{proof}

\begin{lemma}
    \label{lem-broom}
    Let $T$ be the tree with vertex set $\{u,w,v, v_1,v_2,v_3\}$ and edge set $\{uw, wv, vv_1,vv_2,vv_3\}$. Let $f \in \mathbb{N}^T$ be the function defined as $f(w)=f(v)=5m$, $f(u) = f(v_1)=f(v_2)=f(v_3)=3m$. Then $T$ is $(f, 2m)$-choosable.
\end{lemma}
\begin{proof}
    Let $L$ be an $f$-list assignment of $T$. By Lemma \ref{lem-choosable}, we may assume that $L(w)=L(v)$, $L(u), L(v_1), L(v_2), L(v_3) \subseteq L(v)$. 

    Initially, let $\phi(x)= \emptyset$ for all $x \in V(T)$.
    
    Since $|L(v_1)|+|L(v_2)|+|L(v_3)| = 9m$, there is an $m$-subset $A$ of $L(v)$ such that each color in $A$ is contained in at most one of $L(v_1), L(v_2), L(v_3)$. 
    Add $A$ to $\phi(v)$.  

    Let $L'(v_i) = L(v_i) - A$ for $i=1,2,3$. Then $|L'(v_i)|=3m-t_i$, where  $t_i = |A \cap L(v_i)| $. 
 
   Let $t= |A \cap L(u)|$, and let $B'$ be a $t$-subset of $L(v) - (L(u) \cup A)$. 
   Let $D= (L(u)-A) \cup B'$. Thus $|D|=3m$, $|D \cup A|=4m$ and hence $|L(v)-(D \cup A)|=m$. 
  Note that $|D-L'(v_i)| \ge t_i$. 
   Choose an $t_i$-subset $A_i$ of $D-L'(v_i)$, and add $A_i$ to $\phi(v)$. Let $s= |A_1 \cup A_2 \cup A_3|$. Then  $ s \le |A_1|+|A_2|+|A_3|=m$. So now $\phi(v) =m+s \le 2m$. The key property of the set $\phi(v)$ is that 
   \begin{itemize}
       \item[(i)] For $i=1,2,3$, $|\phi(v) - L(v_i)| \ge m$.
       \item[(ii)] $L(v)-(D \cup A) \subseteq  L(v)  - (\phi(v)  \cup L(u))$ and hence $|L(v)  - (\phi(v)  \cup L(u))|    \ge m$.
   \end{itemize}
   Arbitrarily choose $m-s$ colors from 
$D$ and add to $L(v)$. Now $|\phi(v)|=2m$ and  (i) and (ii) still hold. 

It follows from (i) that $|L(v_i) - \phi(v)| \ge 2m$. Let $\phi(v_i)$ be a $2m$-subset of $L(v_i) - \phi(v)$. 
As $L(w)=L(v)$, it follows from (ii) that there is a $2m$-subset $B$ of $L(w)-\phi(v)$ with $|B-L(u)| \ge m$. Let $\phi(w)=B$. Then $|L(u)-\phi(w)| \ge 2m$, and hence $\phi$ can be extended to an $(L,2m)$-coloring of $T$ (by assigning to $u$ a $2m$-subset from $L(u)-\phi(w)$). 
\end{proof}

\section{Sketch of the proof of Theorem \ref{Th:2} and some preparations}
\label{sec-preparation}

  Assume Theorem \ref{Th:2} is not true, and $G \in \mathcal{P}_{4,5}$ is a counterexample with minimum number of vertices. Thus there is a  $7m$-cover $\mathcal{H}=(L,M)$    of $G$  and  an 
$(\mathcal{H}, 2m)$-coloring of $G[D]$ ($D=b(f_0)$ is a good cycle)  that cannot  be extended to an   $(\mathcal{H}, 2m)$-coloring of $G$.  By the minimality of $G$, if $Z$ is a non-empty  set of internal vertices, then $\phi$ can be extended to an $(\mathcal{H}, 2m)$-coloring of $G-Z$. 

Without loss of generality, we may assume that $L(v) = [7m] \times \{v\}$ for each vertex $v$. An edge $e=uv$ is called  \textit{straight} with respect to $\mathcal{H}$ if 
$M_e= \{(i,u)(i,v): i \in [7m]\}$. It is well-known \cite{dvovrak2018correspondence} and easy to see that if $T$ is a tree in $G$, then by a permutation of colors, if needed, we may assume that all edges of $T$ are straight. 

In the next  section, we shall prove some structural property of $G$. The important properties are of the form that $G$ does not have a set $Z$ of internal vertices such that $G[Z]$ induces a special tree $T$, and vertices of $T$ have special neighborhood in $G-Z$.  

To prove this property of $G$, we assume to the contrary that such a set $Z$ exists. 
By the minimality of $G$, there is   an extension $\phi$ of $\phi|_D$ to $G-Z$.

For $v \in Z$, let
$N_{G-Z}(v) = N_G(v) -Z$, and let $d_{G-Z}(v)=|N_{G-Z}(v)|$. 
We denote by $\mathcal{H}^{\phi} = (L', M')$ the cover of $G[Z]$ defined as follows: 
\begin{itemize}
  \item   For $v \in Z$, $L'(v)=L(v) \setminus N_M(\phi(N_{G-Z}(v)))$. 
  \item $M'$ is the restriction of $M$ to $\cup_{v \in Z}L'(v)$.
\end{itemize}

Then $\mathcal{H}^{\phi}$ is an $f$-cover of $G[Z]$, where $f(v) \le  |L'(v)|$. We shall apply   Lemmas in Section \ref{sec-f2m} to prove that $G[Z]$ is $(f, 2m)$-DP-colorable, and hence has an $(\mathcal{H}^{\phi}, 2m)$-coloring $\psi$. The union $\phi \cup \psi$ is then an extension of $\phi|_D$ to $G$, which is a contradiction.

As $L'(v) = L(v) \setminus N_M(\phi(N_{G-Z}(v)))$, and $|L(v)|=7m$, $|\phi(u)|=2m$ for each $u \in N_{G-Z}(v)$, we conclude that $|L'(v)| \ge (7-2 d_{G-Z}(v))m$ for each vertex $v \in Z$. In some cases, we need a larger lower bound for $|L'(z)|$ for some vertex $z \in Z$ so that the lemmas from the last section can be applied. This is achieved as follows:

We choose two non-adjacent vertices $x,y \in N_{G-Z}(z)$. By permute colors if needed, we   assume that both edges $xz, yz$ are straight with respect to $\mathcal{H}$.
We identify $x,y$ into a new vertex $v^*$.  Since $G$ has no 4-cycle, $z$ is the only common neighbor of $x$ and $y$. Hence the identification of $x$ and $y$ does not create  parallel edges. Denote by $G'$ the resulting graph. We denote by $\mathcal{H}=(L,M)$ the ``inherited" $7m$-cover of $G'$, where $L(v^*)=[7m] \times \{v^*\}$, and for each vertex $u \in N_G(x) \cup N_G(y) - Z$,   $(i,v^*) (j,u) \in M_{v^*u}$ if and only if $(i,x)(j,u) \in M_{xu}$ or $(i,y)(j,u) \in M_{yu}$.

If $G' \in \mathcal{P}_{4,5}$ and $G[D]=G'[D]$, then by the minimality of $G$, there is an extension $\phi$ of $\phi|_D$ to $G'$. 

We obtain an extension, also denote by $\phi$, of $\phi|_D$ to $G-Z$ by assigning the color set $\phi(v^*)$ to $x$ and $y$, i.e., let $\phi(x)= \phi(y) = \phi(v^*)$. Since both edges $zx,zy$ are straight with respect to $\mathcal{H}$, we conclude that 
$L(z)-N_M(\phi(x) \cup \phi(y)) = L(z)-N_M(\phi(x))$.  Hence $|L'(z)| \ge (7-2 d_{G-Z}(z)+2)m$.

For the plan above to work, it remains to show that 
$G' \in \mathcal{P}_{4,5}$ and $G[D]=G'[D]$.  The following lemma shows that this is true under some weak condition.

\begin{lemma}
    \label{lem-p45} Let $Z$ be a subset of internal vertices, $z$ is a vertex in $Z$ with $d_G(z)\ge 4$ $x,u,y, v \in G$ are four neighbors of $z$   arranged around $z$ in this cyclic order and $x,y \notin Z$.   If we   identify $x$ and $y$ into a new vertex $v^*$ in $G-Z$, the resulting graph  $G^{\prime} \in \mathcal{P}_{4,5}$ and $G^{\prime}[D]=G[D]$.
\end{lemma}

\begin{proof}
    First we show that $G^{\prime} \in \mathcal{P}_{4,5}$. If $G'$ has a 5-cycle $C$, then $C$ contains $v^*$. Assume $C=  [v^*z_1z_2z_3z_4]$. Then $C'=[xz_1z_2z_3z_4yz]$ is a 7-cycle in $G$. As $G$ has no 4- or 5-cycles, $u,v \notin C'$. Hence $C'$ is a separating good cycle, that separates $u$ from $v$, a contradiction. Similarly $G'$ has no 4-cycles. Hence $G' \in \mathcal{P}_{4,5}$. 
    
    To show that $G'[D]=G[D]$, it amounts to show that 
   \begin{itemize}
    \item  At least one of $x,y$ is not on $D$;
    \item  If one of $\{x,y\}$ lies on $D$, then the other has no neighbor in $D$ (for otherwise $G'$ contains a new edge connecting two vertices of $D$). 
   \end{itemize}
     
    If $x,y\in D$, then $xzy$ is a splitting path of length 2.  By Lemma \ref{lem:8}, $xy$ is an edge of $D$. Then  $C=xzyx$ is a separating 3-cycle which separates $u$ from $z$, a contradiction.
    
    If one of $\{x,y\}$ lies on $D$ then the other has a neighbor in $D$, then $G$ has a splitting path of length 3 containing $xzy$, contradicting Lemma \ref{lem:9}.

    Clearly the boundary of the unbounded face of $G^{\prime}$ is still $D$ with $|D|\le 7$, that is, $D$ is still good with respect to $G^{\prime}$. Namely $G^{\prime}[D]=G[D]$.
\end{proof}

\section{ Reducible configurations}
\label{sec-property}

In this section, we prove a certain substructures cannot be contained in $G$. Such substructures are called reducible configurations.

The following properties of $G$ are trivial:
\begin{enumerate}
    \item $G$ is 2-connected, for otherwise, we can extend the coloring $\phi$ to the blocks of $G$ sequentially.
    \item $G$ has at least one internal vertex, and each internal vertex $v$ has degree at least 3: If $d_G(v) \le 2$, then we can first extend $\phi$ to an $(\mathcal{H},2m)$-coloring of $G-v$, and then extend it to $G$ by assigning to $v$ a set of $2m$ colors in $L(v)$ not adjacent to any colors assigned to its neighbors. 
    \item $G$ has no separating good cycles: If $C$ is a separating cycle with $|C| \le 7$, then we first extend $\phi$ to ${\rm ext}[C]$ (by the minimality of $G$), and then extend the $(\mathcal{H},2m)$-coloring of $C$ to ${\rm int}[C]$. 
    \item $D$ has no chord: Otherwise, the chord separate $D$ into two cycles, and at least one of the cycle is a separating good cycle.
\end{enumerate}

\begin{lemma}
\label{lem:8}
    If $P=xyz$ is a splitting path of $G$, then $xz\in E(D)$, i.e.,  $P$ cuts off a triangle from $G$.
\end{lemma}
\begin{proof}
    Suppose for a contradiction that $xz\notin E(G)$. Let $P_1$ and $P_2$ be the two paths of $D$ between $x$ and $z$. Let $D_i=P_i\cup P$  ($i=1,2$). As $G$ has neither 4- or 5-cycles, $|D_i|\ge6$. It follows that $12\le |D_1|+|D_2|=|D|+4\le 11$, a contradiction.
\end{proof}

\begin{lemma}
\label{lem:9}
    $G$ has no splitting path of length 3.
\end{lemma}
\begin{proof}
    Suppose for a contradiction that $G$ has a splitting path of length 3, say $P=wxyz$ where $w,z\in D$. Let $P_1$ and $P_2$ be the two paths of $D$ between $w$ and $z$. Let $D_i=P_i\cup P$. We may assume that $|D_1| \le|D_2|$. Clearly $|D_1| \ge 4$. As $G$ has no 4- or 5-cycles, $|D_1|\ge 6$. Since $|D_1|+|D_2|=|D|+6 \le 13$, $|D_1|=6$ and $|D_2|\in {6,7}$. As $d_G(x), d_G(y) \ge 3$, $x$ has a neighbor $x^{\prime}$ and $y$ has a neighbor $y'$ not on $P$. If $x'$ or $y'$ is an internal vertex, then one of $D_1,D_2$ is a good separating cycle, a contradiction. If $x', y' \in D$, then $G$ has a 5-cycle, a contradiction. 
\end{proof}

\begin{lemma}
\label{lem:10}
    If $v$ is an internal 4-vertex in $G$, then $v$ has at most two  {pairwise non-adjacent} internal 3-neighbors.
\end{lemma}
\begin{proof}
    Suppose for a contradiction that $v$ has three  {pairwise non-adjacent} internal 3-neighbors, say $v_i,(i=1,2,3)$. Let $Z=\left\{v,v_1,v_2,v_3 \right\}$. Let $\phi$ be an extension of $\phi|_D$  to $G-Z$, $\mathcal{H}^{\phi}$ is an $f$-cover of $G[Z]$, where $f(v)=5m$ and $f(v_1)=f(v_2)=f(v_3)=3m$. By Lemma \ref{lem-claw}, $G[Z]$ has an $(\mathcal{H}^{\phi}, 2m)$-coloring $\phi$. The union $\phi \cup \psi$ is an extension of $\phi|_D$ to $G$, a contradiction.
\end{proof}

\begin{figure}[htp]
    \centering
    \includegraphics[width=13cm]{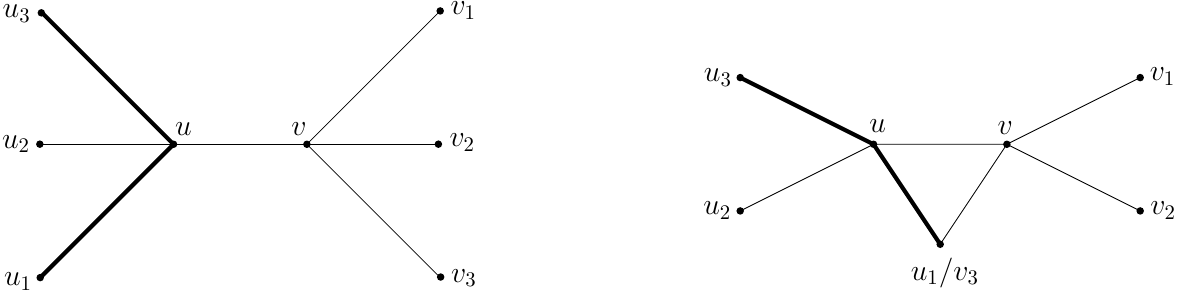}
    \caption{\label{fig:ori2} The configuration in Lemma \ref{lem:11}, where thick line indicate straight edges with respect to $\mathcal{H}$.}
\end{figure}

\begin{lemma}
\label{lem:11}
   If $u, v$ are two adjacent $4$-vertices, the other neighbors of $u$ are $u_1,u_2,u_3$, the other neighbors of $v$ are $v_1,v_2,v_3$, and these vertices are arranged around the edge $uv$ in the cyclic order $u_1,u_2,u_3,v_1,v_2,v_3$, with possibly $u_1=v_3$. Then  
   either $u_2,u_3$ are adjacent, or at least one of $u_2,u_3$ is not an internal 3-vertex. Similarly,  either $v_1,v_2$ are adjacent, or at least one of $v_1,v_2$ is not an internal 3-vertex.
\end{lemma} 

\begin{proof}
    Suppose for a contradiction that $v_1,v_2$ are not adjacent, and both $v_1,v_2$ are internal 3-vertices.  
    
    We may assume that the edges $u_3u, u_1u$ are straight with respect to  $\mathcal{H}$. Let $Z=\{u,v, v_1, v_2\}$. Let $G'$ be obtained from $G-Z$  by  identifying $u_3$ with $u_1$ into a new vertex $u^*$. By  Lemma \ref{lem-p45},   $G^{\prime} \in \mathcal{P}_{4,5}$ and 
    $G'[D] = G[D]$.  Thus $\phi|_D$ has an extension $\phi$ to $G'$.
    
    Note that $\mathcal{H}^{\phi}$ is an $f$-cover of $G[Z]$, where $f(v)=5m$ and $f(u)=f(v_1)=f(v_2)=3m$. 
    By Lemma \ref{lem-claw}, $G[Z]$ has an $(\mathcal{H}^{\phi}, 2m)$-coloring $\psi$. The union $\phi \cup \psi$ is an extension of $\phi|_D$ to $G$, a contradiction. 
\end{proof}

\begin{lemma}
\label{lem:12}
    Let $T=[uvw]$ be a triangle in $G$. If $v$ is a poor 3-vertex or a poor 4-vertex, and $u$  is internal, then $d_G(u) \ge 5$.
\end{lemma}

\begin{proof}
    Assume $u$ is internal and $d_G(u) \le 4$. 

    Assume first that $d_G(u)=3$. If $v$ is a poor 3-vertex, and $v'$ is an isolated internal 3-neighbor of $v$, then let  $Z=\{u,v, v'\}$. If $v$ is a poor 4-vertex, and $v', v''$ are isolated internal 3-neighbors of $v$, then let  $Z=\{u,v, v', v''\}$.  Then $\phi|_D$ can be extended to $G-Z$, and by Lemma \ref{lem-claw}, $G[Z]$ has an $(\mathcal{H}^{\phi}, 2m)$-coloring $\psi$. The union $\phi \cup \psi$ is an extension of $\phi|_D$ to $G$, a contradiction.  
    
   Assume $d_G(u)=4$.  By Lemma \ref{lem:11},  $v$ is not a poor 4-vertex, and hence is a poor 3-vertex. Let $v'$ be the isolated iternal 3-neighbor of $v$ and let  $u_1$ and $u_2$ be the two neighbors of $u$ not on $T$. We may assume that $u_1,u_2,v^{\prime}$ and $w$ are arranged around $uv$ in a cyclic order. Assume edges $uu_2$ and $uw$ are straight with respect to  $\mathcal{H}$.  Let $Z=\{u,v, v'\}$. Let $G'$ be obtained from $G-Z$ by identifying  $u_2$ with $w$ into a new vertex $u^*$.  By Lemma \ref{lem-p45},  $G^{\prime} \in \mathcal{P}_{4,5}$ and $G'[D] = G[D]$. So $\phi$ can be extended from $D$ to $G^{\prime}$. 
     
     Let $\phi(u_2)=\phi(w)=\phi(u^*) $. Then  $\mathcal{H}^{\phi}$ is an $f$-cover of $G[Z]$, where $f(v)= 5m$ and $f(u)=f(v^{\prime})=3m$. By Lemma \ref{lem-claw},  $G[Z]$ has an $(\mathcal{H}^{\phi},2m)$-coloring $\psi$. The union $\phi \cup \psi$ is an extension of $\phi|_D$ to $G$, a contradiction.
\end{proof}

\begin{lemma}
\label{lem:13}
    If $v$ is an internal   5-vertex, then $v$ has at most four isolated internal 3-neighbors.
\end{lemma}

\begin{proof}
     Suppose for a contradiction that $v$ has five isolated internal 3-neighbors   $v_i$ $ (i=1,2,3,4,5)$. Let $Z=\left \{v, v_1,v_2,v_3,v_4,v_5  \right \} $. 
     
     Then there is an extension  $\phi$ of $\phi|_D$ to $G^{\prime}=G-Z$.  Note that $\mathcal{H}^{\phi}$ is an $f$-cover of $G[Z]$, where $f(v)=7m$ and $f(v_i)=3m$ for $i=1,2,\ldots, 5$.  It follows from Lemma \ref{lem-star5} that $G[Z]$ has an $(\mathcal{H}^{\phi}, 2m)$-coloring $\psi$. The union $\phi \cup \psi$ is an extension of $\phi|_D$ to $G$,   a contradiction.
\end{proof}

\begin{lemma}
\label{lem:14}
Let $T=[uvw]$ be a triangle in $G$ and $P=u^{\prime}uv$ an internal path in $G$. If $u$ is poor 3-vertex and  $d_G(v)=5$, then $v$ has at most two internal isolated 3-neighbors.
\end{lemma}
\begin{proof}
    Suppose for a contradiction that $v$ has three internal 3-neighbors, say $v_i, (i=1,2,3)$. Let $u'$ be the isolated internal neighbor of $u$, and let  $Z=\left\{u^{\prime},u,v,v_1,v_2,v_3 \right\}$, $T=G[Z]$. By the minimality of $G$, $\phi$ can be extended from $D$ to $G^{\prime}=G-Z$.  Note that $\mathcal{H}^{\phi}$ is an $f$-cover of $G[Z]$, where   $f(u)=f(v)=5m$ and $f(u^{\prime})=f(v_1)=f(v_2)=f(v_3)=3m$. By Lemma \ref{lem-broom}, $G[Z]$ is DP-$(f,m)$-colorable. Thus $\phi$ can extend from $D$ to $G$, a contradiction. 
\end{proof}

\begin{lemma}
\label{lem:15}
    Let $T=[uvw]$ be a triangle in $G$. If $u$ is a poor 4-vertex and $v$ is an internal 5-vertex, then $v$ has at most two internal isolated 3-neighbors.
\end{lemma}

\begin{proof}
    Suppose for a contradiction that $v$ has three internal isolated 3-neighbors, say $v_i, i=1,2,3$. We may assume that $u,w,v_1,v_2$ and $v_3$ are arranged around $v$ in this cyclic order, and assume that $vv_2, vw$ are straight with respect to  $\mathcal{H}$.
    Let $u_i,i=1,2$, be the two isolated 3-neighbors of $u$. Let $Z=\left \{ u_1,u_2,u,v,v_1,v_3 \right \} $,
    and let $G'$ be obtained from $G-Z$ by identifying  $w$ with $v_2$ into a new vertex $v^*$. 
\par
    By Lemma \ref{lem-p45} and the mimimality of $G$, $G^{\prime}$ has a $(7m,2m)$-DP-coloring which extends $\phi$ from $D$ to $G^{\prime}$. 
    
    Let $\phi(w)=\phi(v_2)=\phi(v^*)$. Then $\mathcal{H}^{\phi}$ is an $f$-cover of $G[Z]$, where $f(u)=f(v)=5m$ and $f(z)=3m$ for $z\in \left \{ u_1,u_2,v_1,v_3 \right \} $. Observe that $G[Z]$ is a double claw. By Lemma \ref{lem-doubleclaw}, $G[Z]$ has an $(\mathcal{H}^{\phi},2m)$-coloring $\psi$. Hence the union $\phi \cup \psi$ is an extension of $\phi|_D$ to $G$, a contradiction completing the proof.
\end{proof}

\section{Discharging}
\label{sec-discharging}

This section completes the proof of Theorem \ref{Th:2} by deriving a contradiction with a discharging procedure. The \textit{initial charge function} $\operatorname{ch}$ is defined as:  $\operatorname{ch}(v)=2 d_G(v)-6$ for $v \in V$, $\operatorname{ch}(f)=d(f)-6$ for $f \in F \backslash \left\{f_{0} \right\}$ and  $\operatorname{ch}(f_{0})=d\left(f_{0}\right)+6$. Then 

\begin{gather*}
    \sum_{x \in V \cup F} \operatorname{ch}(x)=0
\end{gather*}

We apply the following discharging rules to obtain the final charge function  $\operatorname{ch}^{\prime}$.

\subsection{Discharging rules}
\textbf{R1.} Every internal 3-face receives 1 from each of its vertices.
\par
\textbf{R2.} Let $v$ be an internal triangular 3-vertex and $v^{\prime}$ the unique isolated neighbor of $v$.
\par
\quad \quad(1) If $ v $ is poor, then $ v $ receives $ \frac{1}{2} $ from each of its two non-isolated neighbors.
\par
\quad \quad(2) If $ v $ is not poor, then $ v $ receives 1 from $ v^{\prime} $.
\par
\textbf{R3.} Every poor 4-vertex receives $ \frac{1}{2} $ from each of its non-isolated neighbors.
\par
\textbf{R4.} Every external 2-vertex receives 2 from $ f_{0} $.
\par
\textbf{R5.} Let $ v $ be an external 3-vertex.
\par
\quad \quad(1) If $ v $ is triangular, then $ v $ receives $ \frac{3}{2} $ from $ f_{0} $.
\par
\quad \quad(2) If $ v $ is non-triangular, then $ v $ receives 1 from $ f_{0} $.
\par
\textbf{R6.} Let $ v $ be an external 4-vertex.
\par
\quad \quad(1) If $ v $ is 2-triangular, then $ v $ receives 1 from $ f_{0} $.
\par
\quad \quad(2) If $ v $ is 1-triangular, then $ v $ receives $ \frac{1}{2}$ from $f_{0}$.

\subsection{Analysis of the final charges}

We shall prove that   $ \sum_{x \in V \cup F} \operatorname{ch}^{\prime}(x) >0$.
On the other hand,  $\sum_{x \in V \cup F} \operatorname{ch}^{\prime}(x) = \sum_{x \in V \cup F} \operatorname{ch} (x) =0$. This  is a contradiction.  

 For each face $f$,  if $d(f) = 3$,  then by R1, $\operatorname{ch}^{\prime}(f)=\operatorname{ch}(f)+3\times1=d(f)-6 + 3=0$. Otherwise $d(f)\ge 6$, and  $\operatorname{ch}^{\prime}(f)=\operatorname{ch}(f)=d(f)-6\ge 0$.

Next we show that $ch'(v) \ge 0$ for each vertex $v$.
For a vertex $v$, let $v_i$, $i = 1,\ldots,d_G(v)$ be the neighbors of $v$ in a cyclic order, 
and $f_i$, $i = 1,\ldots,d_G(v)$, be the face incident to $v_ivv_{i + 1}$ 
where the indices are taken modulo $d_G(v)$.  If $v$ is a boundary vertex, then let $v_1, v_{d_G(v)}$ be the two boundary neighbor of $v$. The following observation will be frequently used.

\begin{observation}
    \label{ob-triangle}
    If $T=[uvw]$ is a triangle and $v$ is a vertex of degree at least 4, then $T$ receives 1 charge from $v$. By Lemma \ref{lem:12}, at most one of $u,w$ is poor and will receive $1/2$ charge from $v$. Thus the total charge send from $v$ to $T$ and $u,w$ is at most $3/2$.
\end{observation}

\begin{claim}
    For any internal vertex $v$, $\operatorname{ch}^{\prime}(v)\ge 0$.
\end{claim}

We consider four cases.
\begin{enumerate}[label=(\arabic*)]
\item $d_G(v)=3$. 
If $v$ is non-triangular, by our rules, no charge is discharged from or to $v$, hence $\operatorname{ch}^{\prime}(v)=\operatorname{ch}(v)=2d_G(v)-6 = 0$. If $v$ is triangular and $f_1=T=[v_1v_2v]$ is a triangle, then $v_3$ is the unique isolated neighbor of $v$.  So $f_1=T$ receives $1$ from $v$ by R1. If $v$ is poor, then by R2(1), each of $v_1$ and $v_2$ sends $\frac{1}{2}$ to $v$. Hence $\operatorname{ch}^{\prime}(v)=\operatorname{ch}(v)-1 + 2\times\frac{1}{2}=2d_G(v)-6 = 0$.  If $v$ is not poor, then by R2(2), $v_3$ sends $1$ to $v$. Hence $\operatorname{ch}^{\prime}(v)=\operatorname{ch}(v)-1 + 1=2d_G(v)-6 = 0$.
\item $d_G(v)=4$. 
If $v$ is non-triangular, then by Lemma \ref{lem:10}, $v$ has at most two internal $3$-neighbors. 
If $v_i$ is not an internal $3$-neighbor of $v$, then $v_i$ receives nothing from $v$. If $v_i$ is an internal $3$-neighbor of $v$, then $v_i$ receives $1$ or $0$ depending on $v_i$ being triangular or not by rule R2(2). So $\operatorname{ch}^{\prime}(v)\ge \operatorname{ch}(v)-2\times1 = 0$.

If $v$ is $1$-triangular and $f_1=T=[v_1v_2v]$ is a triangle, then by R1, $v$ sends $1$ to $T$. 
Observe that each of $v_3$ and $v_4$ may be a triangular internal $3$-vertex which receives $1$ from $v$ by R2(2). If $v$ is a poor $4$-vertex, by R3, $v$ receives $\frac{1}{2}$ from each of $v_1$ and $v_2$. Hence $\operatorname{ch}^{\prime}(v)\ge \operatorname{ch}(v)-1-2\times1 + 2\times\frac{1}{2}=0$. Otherwise,  at least one of $v_3$ and $v_4$ is not an internal $3$-vertex which receives nothing from $v$. Hence $\operatorname{ch}^{\prime}(v)\ge \operatorname{ch}(v)-1 - 1=0$.

If $v$ is $2$-triangular, say $f_1$ and $f_3$ are t5riangular faces, then  each of $f_1,f_3$ receives $1$ from $v$ by R1. 
By Lemma \ref{lem:12}, each of $v_i$, $i = 1,2,3,4$, is not poor and receives nothing from $v$. So $\operatorname{ch}^{\prime}(v)\ge \operatorname{ch}(v)-1 - 1=0$.

\item $d_G(v)=5$. 
If $v$ is non-triangular, then a neighbor $v_i$ of $v$ receives   $1$ from $v$ by R2(2) if and only if $v_i$ is a triangular $3$-vertex. By Lemma \ref{lem:13}, $v$ has at most four internal $3$-neighbors. 
So $\operatorname{ch}^{\prime}(v)\ge \operatorname{ch}(v)-4\times1=2d_G(v)-6-4\times1 = 0$. 

If $v$ is $1$-triangular, say $f_1=[v_1v_2v]$ is a triangular face incident with $v$, then 
by Observation \ref{ob-triangle}, $f_1$ and $v_1, v_2$ receive at most $3/2$ charge from $v$.

By Lemmas \ref{lem:14} and \ref{lem:15}, $v$ has at most two internal isolated $3$-neighbors,  which may receives $1$ from $v$ by R2(2). So $\operatorname{ch}^{\prime}(v)\ge \operatorname{ch}(v)-\frac{3}{2}-2\times1=\frac{1}{2}>0$. 

If $v$ is $2$-triangular, say $f_1=[v_1v_2v]$ and $f_3=[v_3v_4v]$ are two triangles incident with $v$, then by Observation \ref{ob-triangle}, for each $i \in \{1,3\}$,  $f_i$ together with $v_i, v_{i+1}$ receives at most $3/2$ charge from $v$, and $v_5$  receives at most 1 from $v$ by R2(2). So $\operatorname{ch}^{\prime}(v)\ge  \operatorname{ch}(v)-2\times1 - 2\times\frac{1}{2}-1 = 0$.

\item $d_G(v) \ge 6$. By Observation \ref{ob-triangle}, on the average,
each neighbor of $v$ receives at most 1 from $v$.
So $\operatorname{ch}^{\prime}(v)\ge  \operatorname{ch}(v)-d_G(v)\times1 = 2d_G(v)-6 - d_G(v)=d_G(v)-6\ge 0$.
\end{enumerate}

\begin{claim}
    For any $v\in V(D)$, $\operatorname{ch}^{\prime}(v)\ge 0$.
\end{claim}

\begin{enumerate}[label=(\arabic*)]
    \item 
$d_G(v)=2$.

 By R4, $v$ receives 2 from $f_0$. 
By our rules, no charge is discharged out from a 2-vertex. 
Hence $\operatorname{ch}^{\prime}(v)=\operatorname{ch}(v)+2 = 2d_G(v)-6 + 2 = 0$.
    
  \item $d_G(v)=3$.
  If  $v$ is non-triangular, then by R5(2), $v$ receives 1 from $f_0$. 
     On the other hand, $v$ may have an internal isolated 3-neighbor which may be triangular, 
     hence it receives 1 from $v$ by R2(2). 
     So $\operatorname{ch}^{\prime}(v)\ge \operatorname{ch}(v)+1 - 1 = 0$.
    
    If   $v$ is triangular, then by R5(1), $v$ receives $\frac{3}{2}$ from $f_0$.
    As $G$ has no adjacent triangles, $v$ is incident with exactly one triangle. 
    If $f_0$ is the unique triangle incident with $v$, then each of the other two faces incident with $v$ is an internal $6^+$-face, which receives nothing from $v$, while the unique isolated neighbor of $v$ 
    may be an internal triangular 3-vertex which receives 1 from $v$ by R2(2). 
    Hence giving $\operatorname{ch}^{\prime}(v)\ge \operatorname{ch}(v)+\frac{3}{2}-1=\frac{1}{2}>0$. 
    
    Otherwise, the unique triangle incident with $v$ is internal which receives 1 from $v$ by R1, 
    while the other internal face incident with $v$ is a $6^+$-face which receives nothing from $v$, 
    and the unique internal non-isolated neighbor of $v$ may be poor, hence it receives $\frac{1}{2}$ from $v$ by R2(1) or R3. So $\operatorname{ch}^{\prime}(v)\ge \operatorname{ch}(v)+\frac{3}{2}-1-\frac{1}{2}=0$.
    
   \item  $d_G(v)=4$.
   
   If $v$ is non-triangular, then  $v$ has at most two internal neighbors, 
    each of which may be a poor 3-vertex. By R2(2), $\operatorname{ch}^{\prime}(v)\ge  \operatorname{ch}(v)-2\times1 = 0$.

     If $v$ is 1-triangular, then by  R6(2), $v$ receives $\frac{1}{2}$ from $f_0$. 
     If $f_0$ is the unique triangle incident with $v$, then every internal incident face of $v$ 
     is a $6^+$-face which receives nothing from $v$, while $v$ may have two poor 3-neighbors, 
     each of which receives 1 from $v$ by R2(2). So
     $\operatorname{ch}^{\prime}(v)\ge \operatorname{ch}(v)+\frac{1}{2}-2\times1=\frac{1}{2}>0$. 
     
     Otherwise the only incident triangle of $v$, denoted by $T = [uvw]$, is internal. 
      If both $u$ and $w$ are internal, 
     then the other two internal incident faces of $v$ are $6^+$-faces, 
     each of which receives nothing from $v$, while $T$ receives 1 from $v$ by R1. 
     By Lemma \ref{lem:12}, $T$ has at most one vertex being poor, if any, 
     which receives $\frac{1}{2}$ from $v$ by R2(1) or R3. 
     Hence $\operatorname{ch}^{\prime}(v)\ge \operatorname{ch}(v)+\frac{1}{2}-1-\frac{1}{2}=0$. 
     
     If exactly one of $u$ and $w$, say $u$, is internal, then $T$ receives 1 from $v$ by R1, 
     $u$ receives at most $\frac{1}{2}$ from $v$ by R2(1) or R3. 
     The internal neighbor of $v$ other than $u$ may be a poor 3-vertex, which receives 1 from $v$ 
     by R2(2). Hence $\operatorname{ch}^{\prime}(v)\ge \operatorname{ch}(v)+\frac{1}{2}-1-1- \frac{1}{2}=0$.

    If $v$ is 2-triangular, then by R6(1), $v$ receives 1 from $f_0$. 
    Let $T = [uvw]$ and $T' = [u'v'w']$ be the two triangles incident with $v$. 
    If $f_0$ is a triangle, say $f_0 = T$, then, $u$ and $w$ each receive nothing from $v$ 
    because they all are external, while $T'$ receives 1 from $v$ by R1, and at most one of $u'$ and $w'$
    is poor by Lemma \ref{lem:12}, which receives $\frac{1}{2}$ from $v$ by R2(1) or R3.
    Hence $\operatorname{ch}^{\prime}(v)\ge  \operatorname{ch}(v)+1 - 1-\frac{1}{2}=\frac{3}{2}>0$. 
    
    Otherwise both $T$ and $T'$ are internal and each may have a poor vertex. By R1 and R2(1) or R3, $\operatorname{ch}^{\prime}(v)\ge  \operatorname{ch}(v)+1 - 2\times(1 + \frac{1}{2}) = 0$.

   \item  $d_G(v)\ge 5$.

   By Observation \ref{ob-triangle}, for each triangular face $f$ of $v$, the charge send from $v$ to $f$ and the two neighbors of $v$ on the boundary of $f$ is at most $3/2$. To each other internal neighbor $v_i$ of $v$, the charge send from $v$ to $v_i$ is at most $1$. The two boundary neighbor of $v$ receive no charge from $v$. 
However, if a boundary neighbor, say $v_1$, is contained in a triangle $T=[vv_1v_2]$, then the total charge send from $v$ to $T$ and $v_2$ is at most $3/2$. Therefore, the total charge send out from $v$ is at most $d_G(v)-4 + 2 \times \frac 32 = d_G(v)-1$. So $ch'(v)  \ge ch(v)-d_G(v)+1 = d_G(v)-5 \ge 0$. 
\end{enumerate}

\begin{claim}
    $\operatorname{ch}^{\prime}(f_0)>0$.
\end{claim}

As $G$ is 2-connected and not a cycle, $f_0$ has at least two $3^+$-vertices. If $f_0$ is a triangle, then $f_0$ has at most one $2$-vertex, if any, which receives 2 from $f_0$ by R4, while each of the other two vertices receives at most $\frac{3}{2}$ from $f_0$ by R4-R6. Hence  $\operatorname{ch}^{\prime}(f_0)\ge  \operatorname{ch}(f_0)-2 - 2\times\frac{3}{2}= \operatorname{ch}(f_0)+6 - 5 = 4>0$. 

 If $f_0$ has no triangular edges, then $f_0$ has at least two $3^+$-vertices, each of which receives at 
 most 1 from $f_0$ by R5-R6. Hence  $\operatorname{ch}^{\prime}(f_0)\ge  \operatorname{ch}(f_0)-(d(f_0)-2)\times 2 - 2\times 1=8 - d(f_0)\ge  1>0$. 
 
 Assume that $f_0$ has at least one triangular edge. Observe that if $uv$ is an triangular edge, then the sum of the charge discharged from $f_0$ to $u$ and to $v$ is at most 3 by R5-R6. 
 
 If $f_0$ has at least two triangular edges, then $\operatorname{ch}^{\prime}(f_0)\ge  \operatorname{ch}(f_0)-(d(f_0)-4 )\times 2+2\times 3=8 - d(f_0)>0$. 
 
 Assume $f_0$ has exactly one triangular edge, say $uv$. If at least one of $u$ and $v$ is a $4^+$-vertex, then the sum of charge discharged from $f_0$ to $u$ and $v$ is at most  $\frac{3}{2}+\frac{1}{2}=2$  by R5 and R6(2).
Hence $\operatorname{ch}^{\prime}\left(f_{0}\right) \geq \operatorname{ch}\left(f_{0}\right)-\left(d\left(f_{0}\right)-2\right) \times 2-2=8-d\left(f_{0}\right)>0 $. 
If both $u$ and $ v $ are 3 -vertices, then as $ G $ is 2-connected, $f_{0}$ has at least one non-triangular $3^{+}$-vertex which receives at most 1 from $f_{0} $. Hence $\operatorname{ch}^{\prime}\left(f_{0}\right) \geq \operatorname{ch}\left(f_{0}\right)-\left(d\left(f_{0}\right)-3\right) \times 2-3-1=8-d\left(f_{0}\right)>0$.

This completes the proof of Theorem \ref{Th:2}.

\bibliographystyle{plain}
\bibliography{Reference} 

@article{bernshteyn2020fractional,
  title={Fractional DP-colorings of sparse graphs},
  author={Bernshteyn, Anton and Kostochka, Alexandr and Zhu, Xuding},
  journal={Journal of Graph Theory},
  volume={93},
  number={2},
  pages={203--221},
  year={2020},
  publisher={Wiley Online Library}
}

@article{dvovrak2018correspondence,
  title={Correspondence coloring and its application to list-coloring planar graphs without cycles of lengths 4 to 8},
  author={Dvo{\v{r}}{\'a}k, Zden{\v{e}}k and Postle, Luke},
  journal={Journal of Combinatorial Theory, Series B},
  volume={129},
  pages={38--54},
  year={2018},
  publisher={Elsevier}
}

@article {XuZhuStrongfractionalchoicenumber,
    AUTHOR = {Xu, Rongxing and Zhu, Xuding},
     TITLE = {The strong fractional choice number and the strong fractional
              paint number of graphs},
   JOURNAL = {SIAM J. Discrete Math.},
  FJOURNAL = {SIAM Journal on Discrete Mathematics},
    VOLUME = {36},
      YEAR = {2022},
    NUMBER = {4},
     PAGES = {2585--2601},
      ISSN = {0895-4801,1095-7146},
   MRCLASS = {05C15 (05C72)},
  MRNUMBER = {4506573},
MRREVIEWER = {Xiaolan\ Hu},
       DOI = {10.1137/21M1434556},
       URL = {https://doi.org/10.1137/21M1434556},
}

@article {DvorakHu11/3,
    AUTHOR = {Dvo\v{r}\'{a}k, Zden\v{e}k and Hu, Xiaolan},
     TITLE = {Planar graphs without cycles of length 4 or 5 are
              {$(11:3)$}-colorable},
   JOURNAL = {European J. Combin.},
  FJOURNAL = {European Journal of Combinatorics},
    VOLUME = {82},
      YEAR = {2019},
     PAGES = {102996, 18},
      ISSN = {0195-6698,1095-9971},
   MRCLASS = {05C15 (05C10)},
  MRNUMBER = {3983124},
MRREVIEWER = {Deming\ Li},
       DOI = {10.1016/j.ejc.2019.07.007},
       URL = {https://doi.org/10.1016/j.ejc.2019.07.007},
}

@article {MR3004485,
    AUTHOR = {Borodin, O. V.},
     TITLE = {Colorings of plane graphs: {A} survey},
   JOURNAL = {Discrete Math.},
  FJOURNAL = {Discrete Mathematics},
    VOLUME = {313},
      YEAR = {2013},
    NUMBER = {4},
     PAGES = {517--539},
      ISSN = {0012-365X,1872-681X},
   MRCLASS = {05C10 (05C15)},
  MRNUMBER = {3004485},
       DOI = {10.1016/j.disc.2012.11.011},
       URL = {https://doi.org/10.1016/j.disc.2012.11.011},
}

@article {MR116319,
    AUTHOR = {Gr\"otzsch, Herbert},
     TITLE = {Zur {T}heorie der diskreten {G}ebilde. {VI}. {E}in
              {K}antentransformationssatz f\"ur gerade {D}reikantnetze mit
              {V}iereckssystem auf der {K}ugel},
   JOURNAL = {Wiss. Z. Martin-Luther-Univ. Halle-Wittenberg Math.-Natur.
              Reihe},
  FJOURNAL = {Wissenschaftliche Zeitschrift der Martin-Luther-Universit\"at
              Halle-Wittenberg. Mathematisch-Naturwissenschaftliche Reihe},
    VOLUME = {7},
      YEAR = {1958},
     PAGES = {447--456},
      ISSN = {0138-1504},
   MRCLASS = {55.00},
  MRNUMBER = {116319},
}

@article {MR1148923,
    AUTHOR = {Abbott, H. L. and Zhou, B.},
     TITLE = {On small faces in {$4$}-critical planar graphs},
   JOURNAL = {Ars Combin.},
  FJOURNAL = {Ars Combinatoria. A Canadian Journal of Combinatorics},
    VOLUME = {32},
      YEAR = {1991},
     PAGES = {203--207},
      ISSN = {0381-7032,2817-5204},
   MRCLASS = {05C38},
  MRNUMBER = {1148923},
}

@article {MR1326930,
    AUTHOR = {Sanders, Daniel P. and Zhao, Yue},
     TITLE = {A note on the three color problem},
   JOURNAL = {Graphs Combin.},
  FJOURNAL = {Graphs and Combinatorics},
    VOLUME = {11},
      YEAR = {1995},
    NUMBER = {1},
     PAGES = {91--94},
      ISSN = {0911-0119,1435-5914},
   MRCLASS = {05C15},
  MRNUMBER = {1326930},
MRREVIEWER = {Chinh\ T.\ Hoang},
       DOI = {10.1007/BF01787424},
       URL = {https://doi.org/10.1007/BF01787424},
}

@article {MR2117940,
    AUTHOR = {Borodin, O. V. and Glebov, A. N. and Raspaud, A. and
              Salavatipour, M. R.},
     TITLE = {Planar graphs without cycles of length from 4 to 7 are
              3-colorable},
   JOURNAL = {J. Combin. Theory Ser. B},
  FJOURNAL = {Journal of Combinatorial Theory. Series B},
    VOLUME = {93},
      YEAR = {2005},
    NUMBER = {2},
     PAGES = {303--311},
      ISSN = {0095-8956,1096-0902},
   MRCLASS = {05C15},
  MRNUMBER = {2117940},
MRREVIEWER = {H.\ L.\ Abbott},
       DOI = {10.1016/j.jctb.2004.11.001},
       URL = {https://doi.org/10.1016/j.jctb.2004.11.001},
}

@article {MR3575214,
    AUTHOR = {Cohen-Addad, Vincent and Hebdige, Michael and Kr\'al', Daniel
              and Li, Zhentao and Salgado, Esteban},
     TITLE = {Steinberg's conjecture is false},
   JOURNAL = {J. Combin. Theory Ser. B},
  FJOURNAL = {Journal of Combinatorial Theory. Series B},
    VOLUME = {122},
      YEAR = {2017},
     PAGES = {452--456},
      ISSN = {0095-8956,1096-0902},
   MRCLASS = {05C15 (05C10)},
  MRNUMBER = {3575214},
       DOI = {10.1016/j.jctb.2016.07.006},
       URL = {https://doi.org/10.1016/j.jctb.2016.07.006},
}

@article {Voigt2007,
    AUTHOR = {Voigt, M.},
     TITLE = {A non-3-choosable planar graph without cycles of length 4 and
              5},
   JOURNAL = {Discrete Math.},
  FJOURNAL = {Discrete Mathematics},
    VOLUME = {307},
      YEAR = {2007},
    NUMBER = {7-8},
     PAGES = {1013--1015},
      ISSN = {0012-365X,1872-681X},
   MRCLASS = {05C15},
  MRNUMBER = {2297188},
MRREVIEWER = {Klas\ Markstr\"om},
       DOI = {10.1016/j.disc.2005.11.041},
       URL = {https://doi.org/10.1016/j.disc.2005.11.041},
}

@article {KangJinZhu,
    AUTHOR = {Kang, Yingli and Jin, Ligang and Zhu, Xuding},
     TITLE = {Planar graphs having no cycle of length 4, 7, or 9 are
              {DP}-3-colorable},
   JOURNAL = {J. Graph Theory},
  FJOURNAL = {Journal of Graph Theory},
    VOLUME = {107},
      YEAR = {2024},
    NUMBER = {2},
     PAGES = {344--358},
      ISSN = {0364-9024,1097-0118},
   MRCLASS = {05C10 (05C15 05C38)},
  MRNUMBER = {4788453},
MRREVIEWER = {Deming\ Li},
       DOI = {10.1002/jgt.23123},
       URL = {https://doi.org/10.1002/jgt.23123},
}

@incollection {AlonTuzaVoigt,
    AUTHOR = {Alon, N. and Tuza, Zs. and Voigt, M.},
     TITLE = {Choosability and fractional chromatic numbers},
      NOTE = {Graphs and combinatorics (Marseille, 1995)},
   JOURNAL = {Discrete Math.},
  FJOURNAL = {Discrete Mathematics},
    VOLUME = {165/166},
      YEAR = {1997},
     PAGES = {31--38},
      ISSN = {0012-365X,1872-681X},
   MRCLASS = {05C15},
  MRNUMBER = {1439258},
       DOI = {10.1016/S0012-365X(96)00159-8},
       URL = {https://doi.org/10.1016/S0012-365X(96)00159-8},
}

@article {ZhuStrongfractionalchoicenumber,
    AUTHOR = {Zhu, Xuding},
     TITLE = {Multiple list colouring of planar graphs},
   JOURNAL = {J. Combin. Theory Ser. B},
  FJOURNAL = {Journal of Combinatorial Theory. Series B},
    VOLUME = {122},
      YEAR = {2017},
     PAGES = {794--799},
      ISSN = {0095-8956,1096-0902},
   MRCLASS = {05C15 (05C10)},
  MRNUMBER = {3575229},
MRREVIEWER = {Daniel\ Kr\'al},
       DOI = {10.1016/j.jctb.2016.09.008},
       URL = {https://doi.org/10.1016/j.jctb.2016.09.008},
}

@article {WangYingQian,
    AUTHOR = {Wang, Yingqian},
     TITLE = {Planar graphs with neither 4-cycles nor 5-cycles   are
$(7 : 2)$-colorable
},
   JOURNAL = {Discrete Mathematics},
  FJOURNAL = {Discrete Mathematics},
      YEAR = {2025, to appear},   
}

@article {BernshteinYuKostochkaPron,
    AUTHOR = {Bernshte\u{i}n, A. Yu. and Kostochka, A. V. and Pron', S.
              P.},
     TITLE = {On {DP}-coloring of graphs and multigraphs},
   JOURNAL = {Sibirsk. Mat. Zh.},
  FJOURNAL = {Rossi\u iskaya Akademiya Nauk. Sibirskoe Otdelenie. Institut
              Matematiki im. S. L. Soboleva. Sibirski\u i\ Matematicheski\u
              i\ Zhurnal},
    VOLUME = {58},
      YEAR = {2017},
    NUMBER = {1},
     PAGES = {36--47},
      ISSN = {0037-4474},
   MRCLASS = {05C15 (05C69 05C75)},
  MRNUMBER = {3686937},
MRREVIEWER = {Zden\v ek\ Ryj\'a\v cek},
       DOI = {10.1134/s0037446617010049},
       URL = {https://doi.org/10.1134/s0037446617010049},
}

\end{document}